\documentclass[11pt,twoside]{article}
\usepackage{pdfsync}
\usepackage{mathrsfs,amssymb,amsfonts,amsthm,amsmath,amsbsy}
\usepackage[numbers,sort,square]{natbib}
\usepackage{dsfont,verbatim,fancyhdr}
\usepackage{graphicx,float} 
\usepackage[hang]{caption} 
\usepackage[usenames,dvipsnames]{xcolor}
\usepackage[bookmarks,bookmarksnumbered,colorlinks=true,pdfstartview=FitV, linkcolor=Red,citecolor=blue!90!black,urlcolor=blue!90!black]{hyperref}
\usepackage[normalem]{ulem}
\usepackage{autonum}
\usepackage[ruled,vlined]{algorithm2e}
\usepackage{autonum}
\graphicspath{{../figures/}}

\usepackage{autonum}

\newif\ifpageone

\allowdisplaybreaks

\def \d {\mathrm{d}}

\def \A {\mathcal{A}}
\def \AA {\mathcal{A}}
\def \B {\mathcal{B}}
\def \C {\mathcal{C}}

\def \M {\mathcal{M}}
\def \F {\mathcal{F}}
\def \P {\mathcal{P}}

\def \X {\mathcal{X}}
\def \Y {\mathcal{Y}}

\def\Z{\mathbb{Z}}

\def \oo {\mathbf{0}}
\def \ee {\mathbf{e}}

\def \mm {\mathbf{m}}

\def \MM {\mathbf{M}}
\def \nn {\mathbf{n}}

\def \xx {\mathbf{x}}

\usepackage{bm}
\def\aalpha{\bm{\alpha}} 

\newcommand{\R}{\mathbb{R}}
\renewcommand{\P}{\mathbb{P}}
\newcommand{\E}{\mathbb{E}}

\renewcommand{\l}{\left}
\renewcommand{\r}{\right}

\newcommand{\aw}{\overset{\leftarrow}{\omega}}
\newcommand{\avt}{\overset{\leftarrow}{\vartheta}}
\newcommand{\aMM}{\overset{\leftarrow}{M}}

\newtheorem{theorem}{Theorem}[section]
\newtheorem{proposition}[theorem]{Proposition} 
\newtheorem{corollary}[theorem]{Corollary}

\newtheorem{definition1}[theorem]{Definition} 
\newtheorem{remark1}[theorem]{Remark} 

\long\def\symbolfootnote[#1]#2{\begingroup\def\thefootnote{\hspace*{-1mm}\fnsymbol{footnote}}\footnote[#1]{#2}\endgroup}

\title{\bf 
Approximate filtering via discrete dual processes
}
\author{\normalsize
\textsc{Guillaume Kon Kam King},
\emph{\normalsize Universit\'e Paris-Saclay - INRAE}\\[2mm]
\textsc{\normalsize Andrea Pandolfi},
\emph{\normalsize Bocconi University}\\[2mm]
\textsc{\normalsize Marco Piretto},
\emph{\normalsize BrandDelta}\\[2mm]
\textsc{\normalsize Matteo Ruggiero\footnote{Corresponding author: ESOMAS Dept., Corso Unione Sovietica 218/bis, 10134, Torino, Italy; matteo.ruggiero@unito.it}},
\emph{\normalsize University of Torino and Collegio Carlo Alberto}}
\date{\today}

\begin{document}
\maketitle
\thispagestyle{empty}

\begin{quote}
\footnotesize

We consider the task of filtering a dynamic parameter evolving as a diffusion process, given data collected at discrete times from a likelihood which is conjugate to the reversible law of the diffusion, when a generic dual process on a discrete state space is available. Recently, it was shown that duality with respect to a death-like process implies that the filtering distributions are finite mixtures, making exact filtering and smoothing feasible through recursive algorithms with polynomial complexity in the number of observations. Here we provide general results for the case where the dual is a regular jump continuous-time Markov chain on a discrete state space, which typically leads to filtering distribution given by countable mixtures indexed by the dual process state space. We investigate the performance of several approximation strategies on two hidden Markov models driven by Cox--Ingersoll--Ross and Wright--Fisher diffusions, which admit duals of birth-and-death type, and compare them with the available exact strategies based on death-type duals and with bootstrap particle filtering on the diffusion state space as a general benchmark.

\end{quote}


\section{Introduction}

Hidden Markov models are widely used statistical models for time series that assume an unobserved Markov process $(X_{t})_{t\geq 0}$, or hidden \emph{signal}, driving the process that generates the observations $(Y_{t_i})_{i=0,\dots, n}$, e.g., by specifying the dynamics of one or more parameters of the observation density $f_{X_{t}}(y)$, called \emph{emission distribution}. See \cite{CMR05} for a general treatment of hidden Markov models. In this framework, the first task is to estimate the trajectory of the signal given observations collected at discrete times $0=t_{0}<t_{1}<\cdots<t_{n}=T$, which amounts to performing sequential Bayesian inference by computing the so-called \emph{filtering distributions} $p(x_{t_{i}}|y_{t_{0}},\ldots,y_{t_{i-1}})$, i.e., the marginal distributions of the signal at time $t_{i}$ conditional on observations collected up to time $t_{i-1}$. Originally motivated by real-time tracking and navigation systems and pioneered by \cite{K60,KB61}, classical and widely known explicit results for this problem include: the \emph{Kalman--Bucy} filter, when both the signal and the observation process are formulated in a gaussian linear system; the \emph{Baum--Welch} filter, when the signal has a finite state-space as the observations are categorical; the \emph{Wonham} filter, when the signal has a finite state-space and the observations are Gaussian. These scenarios allow the derivation of so-called \emph{finite-dimensional filters}, i.e., a sequence of filtering distributions whose explicit identification is obtained through a parameter update based on the collected observations and on the time intervals between the collection times. In such case, the resulting computational cost of the algorithm increases linearly with the number of observations. 
Other explicit results include \cite{S81, FR90, FV98, RS01, G03, GK04, CGK11}. Outside these classes, explicit solutions are difficult to obtain, and their derivation typically relies on \emph{ad hoc} computations. This is especially true when the map $x\mapsto f_{x}$ is non-linear and when the signal transition kernel is known up to a series expansion, often intractable, as is the case for many widely used stochastic models. When exact solutions are not available, one must typically make use of approximate strategies, whose state of the art is most prominently based on extensions of the Kalman and particle filters. See, for example, \cite{BC09,CP20}.

A somewhat weaker but useful notion with respect to that of a finite-dimensional filter was formulated in \cite{CG06}, who introduced the concept of \emph{computable filter}. This extends the former class to a larger class of filters whose marginal distributions are \emph{finite mixtures} of elementary kernels, rather than single kernels. Unlike the former case, such a scenario entails a higher computational cost, usually polynomial in the number of observations, but avoids the infinite-dimensionality typically implied by series expansion of the signal transition kernel. See \cite{CG06,CG09} for two examples.

Recently, \cite{PR14} derived sufficient conditions for computable filtering based on duality. A \emph{dual process} is a process $D_{t}$ which enjoys the identity 
\begin{equation}\label{duality identity}
\E[h(X_{t},d)|X_{0}=x]=\E[h(x,D_{t})|D_{0}=d].
\end{equation} 
Here the expectation on the left-hand side is taken with respect to the transition law of the signal $X_{t}$, and that on the right hand side with respect to that of $D_{t}$, while the class of functions $h(x,d)$ which satisfy the above identity are called duality functions. See \cite{JK14} for a review and for the technical details we have overlooked here. The use of duality is largely established in probability and statistical physics, see for example \cite{EK86,M99,BEG00,Gea07,EG09,Gea09a,Gea09b,Eea10,O10,Cea15,Fea21,Gea22,AS05,Bea09, Dea23, HM11, Fea18, HW07}.
The use of duality for inference was initiated by \cite{PR14}, who showed that for a reversible signal whose marginal distributions are conjugate to the emission distribution (i.e., the Bayesian update at a fixed $t$ can be computed in closed-form), computable filtering is guaranteed if the stochastic part of the dual process evolves on a finite state space. Examples of such scenario were given for non-linear hidden Markov models involving signals driven by Cox--Ingersoll--Ross (CIR) and $K$-dimensional Wright--Fisher (WF) diffusions, for which recursive formulae for the filtering distributions were derived. Along similar lines, duality was exploited for computable \emph{smoothing} in \cite{KKK21}, whereby the signal is also conditioned on data collected at later times, and for nonparametric hidden Markov models driven by Fleming--Viot and Dawson--Watanabe diffusions in \cite{PRS16, ALR21, ALR22}.


In this paper, we investigate filtering problems for hidden Markov models when the dual process takes the more general form of a continuous-time Markov chain on a discrete state space. This is of interest for example in some population genetic models with selection \citep{BEG00} or interaction \cite{Aea19,Fea21} whose known dual processes are of birth-and-death (B\&D) type, whose specific filtering problems are currently under investigation by some of the authors. When the dual process evolves in a countable state space, the filtering distributions can in general be expected to be countably infinite mixtures. This leads to inferential procedures which are not \emph{computable} in the sense specified above, since the computation of the filtering distribution can no longer be exact. 
It is thus natural to wonder how the inferential procedures obtained in such duality-based scenario, possibly aided by some suitable approximation strategies, would perform. 

The paper is organized as follows. 
In Section \ref{sec:conditions} we identify sufficient conditions for filtering based on discrete duals and provide a general description of the filtering operations in this setting. In Section \ref{sec:algorithms}, we apply these results to devise practical algorithms which allow to evaluate in recursive form filtering and smoothing distributions under this formulation.  
Section \ref{sec:CIR} and \ref{sec:WF} investigate hidden Markov models driven by a Cox--Ingersoll--Ross diffusion,  which admits a dual given by a one-dimensional B\&D  process, and by a $K$-dimensional Wright--Fisher diffusions, which is shown to admit a dual given by a $K$-dimensional Moran model. The latter can be seen as a multidimensional B\&D process with constant total population size. Section \ref{sec:experiments} discusses several approximation strategies used to implement the above algorithms with these dual processes, and compares their performance with exact filtering based on the results in \cite{PR14} and with a bootstrap particle filter as a general benchmark. Finally, we conclude with some brief remarks.


\section{Filtering via discrete dual processes}\label{sec:conditions}

Consider a hidden signal $(X_{t})_{t\ge0}$ given by a diffusion process on $\X\subset \R^{K}$, for $K\ge1$. 
This takes here the role of a temporally evolving parameter which is the main target of estimation.
 Observations $Y_{t_{i}}\in \Y\subset \R^{D}$, $D\ge1$, are collected at discrete times $0=t_{0}<t_{1}<\cdots<t_{n}=T$ with distribution $Y_{t_{i}}\overset{\text{ind}}{\sim} f_{x}(\cdot)$, given $X_{t_{i}}=x$. Knowledge of the signal state $x$ thus makes the observations conditionally independent.  Given an observation $Y=y$, define the \emph{update operator} $\phi_y$ acting on densities $\xi$ on $\X$ by
\begin{equation}\label{update operator}
\phi_y(\xi)(x) := \frac{f_x(y)\xi(x)}{\mu_{\xi}(y)}, \qquad 
\mu_{\xi}(y) := \int_{\mathcal{X}} f_x(y)\xi(x).
\end{equation}
Here and later we assume all densities of interest exist with respect to an appropriate dominating measure. In \eqref{update operator}, $\xi$ acts as a \emph{prior} distribution on the signal state, which encodes the current knowledge (or lack thereof) on $X_{t}$, whereas $\mu_{\xi}(y)$ is interpreted as the marginal likelihood of a data point $y$ when $X_{t}$ has distribution $\xi$. The update operator thus amounts to an application of Bayes' theorem for conditioning $\xi$ on a new observation $y$, leading to the updated density $\phi_y(\xi)$. For example, if $\xi$ is a $\text{Beta}(a,b)$ density on $[0,1]$, and $f_{x}$ is $\text{Bern}(x)$, then $\phi_y(\xi)$ is $\text{Beta}(a+y,b+1-y)$ as in a classical Beta-Bernoulli update.

Define also the \emph{propagation operator} $\psi_t$ by
\begin{equation}\label{propagation operator}
\psi_t(\xi)(x') := 
\int_{\mathcal{X}}\xi(x)P_t(x'|x)\d x'.
\end{equation} 
where $P_{t}$ is the transition density of the signal. Here $\psi_t(\xi)$ is the probability density at time $t$ obtained by propagating forward the density $\xi$ of the signal at time 0 by means of the signal semigroup. 

We will make three assumptions, the first two of which are the same as in \cite{PR14}. 

\begin{quote}
\textbf{Assumption 1} (Reversibility). The signal $X_{t}$ is reversible with respect to the density $\pi$, i.e., $\pi(x)P_t(x'|x)=\pi(x')P_t(x|x')$. 
\end{quote}

See the discussion in \cite{PR14} on the possibility of relaxing the above assumption. 
For $K\in \Z_{+}$, define now the space of multi-indices $\M=\Z_{+}^{K}$ to be
\begin{equation}\label{space M}
\M=\{\mm =(m_{1},\ldots,m_{K}):\ m_{j}\in\Z_{+},\text{ for }j=1,\ldots,K\},
\end{equation} 
whose origin is denoted $\oo=(0,\ldots,0)$. 

\begin{quote}
\textbf{Assumption 2} (Conjugacy). For $\Theta\subset \R^{l}$, $l\in \Z_{+}$, let $h:\X\times \M\times \Theta\rightarrow \R_{+}$ be such that $\sup_{x\in \X}h(x,\mm,\theta)<\infty$ for all $\mm\in \M,\theta\in \Theta$ and $h(x,\oo,\theta')=1$ for some $\theta'\in \Theta$. Then $f_{x}(\cdot)$ is conjugate to distributions in the family
\begin{equation}\label{family F}\nonumber
\F=\{g(x,\mm,\theta)=h(x,\mm,\theta)\pi(x),\ \mm\in \M,\theta\in\Theta\},
\end{equation} 
i.e., there exist functions $t:\Y\times \M\rightarrow \M$ and $T:\Y\times \Theta\rightarrow\Theta$ such that if $X_{t}\sim g(x,\mm,\theta)$ and $Y_{t}|X_{t}=x\sim f_{x}$, we have $X_{t}|Y_{y}=y\sim g(x,t(y,\mm),T(y,\theta))$.
\end{quote}

Here $g(x,\mm,\theta)$ takes the role of ``current'' prior distribution, i.e., the prior information on the signal state, possibly based on past observations through previous conditioning and propagations,  and $g(x,t(y,\mm),T(y,\theta))$ takes the role of the posterior, i.e., $g(x,\mm,\theta)$ conditional on a data point $y$ observed from $f_{x}$ for $X_{t}=x$. The functions $t(y,\mm),T(y,\theta)$ provide the transformations that update the parameters based on $y$. In absence of data, the condition $h(x,\oo,\theta')=1$ reduces $g(x,\mm,\theta)$ to $\pi(x)$.

The third assumption weakens Assumption 3 in \cite{PR14} by assuming the dual process has finite activity on a discrete state space, and possibly has a deterministic companion.

\begin{quote}
\textbf{Assumption 3} (Duality). Given a deterministic process $\Theta_{t}\in \Theta$ and a regular jump continuous-time Markov chain $M_{t}$ on $\Z_{+}^{K}$ with transition probabilities
\begin{equation}\label{transition probabilities}
p_{\mm,\nn}(t;\theta):=\P(M_{t}=\nn |  M_{0}=\mm,\Theta_{0}=\theta),
\end{equation}
equation \eqref{duality identity} holds with $D_{t}=(M_{t},\Theta_{t})$ and $h$ as in Assumption 2.
\end{quote}

The following result provides a full description of the propagation and update steps which allow to compute the filtering distribution.

\begin{proposition}\label{prop:prop&update}
Let  Assumptions 1-3 hold, and let $\sum_{\mm\in \M}w_{\mm}g(x,\mm,\theta)$ be a countable mixture with $\sum_{\mm\in \M}w_{\mm}=1$. Then, for $\psi_t$ as in \eqref{propagation operator} we have
\begin{equation}\label{forward propagation}
\psi_t
\bigg(\sum_{\mm\in \M}w_{\mm}g(x,\mm,\theta)\bigg)=\sum_{\nn\in \M}w'_{\nn}(t)g(x,\nn,\Theta_{t}),
\end{equation} 
where
\begin{equation}\label{weights rearrangement after propagation}
\begin{aligned}
w'_{\nn}(t)=\sum_{\mm\in \M}w_{\nn}p_{\mm,\nn}(t;\theta), \quad 
\end{aligned}
\end{equation} 
and $p_{\mm,\nn}(t;\theta)$
are as in \eqref{transition probabilities}. Furthermore, for $\phi_{y}$ as in \eqref{update operator}, we have
\begin{equation}\label{update operation}
    \phi_{y} \bigg(\sum_{\mm\in \M}w_{\mm}g(x,\mm,\theta)\bigg)=
    \sum_{\mm\in \M}\hat w_{\nn,\theta}(y)g(x,t(y,\mm),T(y,\theta))
\end{equation} 
where $\hat w_{\mm,\theta}(y)\propto w_{\mm}\mu_{\mm,\theta}(y)$
and
\begin{equation}\label{marginals}
\begin{aligned}
\mu_{\mm,\theta}(y):=\int_{\X}f_{x}(y)g(x, \mm, \theta)\d x.
\end{aligned}
\end{equation} 
\end{proposition}
\begin{proof}
First observe that $\psi_t(g(x,\mm,\theta))=\sum_{\nn\in \M}p_{\mm,\nn}(t;\theta)g(x,\nn,\Theta_{t})$, which follows similarly to Proposition 2.2 in \cite{PR14}. Then the claim follows by linearity using the fact that
\begin{equation}\label{oper_prop}
\psi_t \bigg( \sum_{i\ge1}w_i\xi_i \bigg)  = \sum_{i\ge1} w_i \psi_t (\xi_i)
\end{equation}
so that
\begin{equation}\nonumber
\begin{aligned}
\psi_t &\,\bigg(\sum_{\mm\in \M}w_{\mm}g(x,\mm,\theta)\bigg)
=
\sum_{\mm\in \M}w_{\mm}\psi_{t}(g(x,\mm,\theta))\\
=&\,
\sum_{\mm\in \M}w_{\mm}\sum_{\nn\in \M}p_{\mm,\nn}(t;\theta)g(x,\nn,\Theta_{t})
=
\sum_{\nn\in \M}\sum_{\mm\in \M}w_{\mm}p_{\mm,\nn}(t;\theta)g(x,\nn,\Theta_{t})
\end{aligned}
\end{equation} 
Using now the fact that
\begin{equation}\label{mixture update}
\phi_y \bigg( \sum_{i\ge1}w_i\xi_i\bigg) = \sum_{i\ge1}      \frac{w_i \mu_{\xi_i}(y)}{\sum_j w_j \mu_{\xi_j}(y)} \phi_y(\xi_i),\qquad 
\end{equation} 
we also find that
\begin{equation}
    \phi_{y} \bigg(\sum_{\mm\in \M}w_{\mm}g(x,\mm,\theta)\bigg)=
    \sum_{\mm\in \M}\frac{w_{\mm}\mu_{\mm,\theta}(y)}{\sum_{\nn\in \M}w_{\nn}\mu_{\nn,\theta}(y)}g(x,t(y,\mm),T(y,\theta))
\end{equation}
where
\begin{equation}
\mu_{\mm,\theta}(y)=
\int_{\mathcal{X}} f_x(y)g(x,\mm,\theta)\d x
=\int_{\mathcal{X}} f_x(y)h(x,\mm,\theta)\pi(x)\d x.
\end{equation}
\end{proof}

The expression \eqref{forward propagation}, together with \eqref{weights rearrangement after propagation}, provides a general recipe on how to compute the forward propagation of the current marginal distribution of the signal $g(x,\mm,\theta)$, based on the transition probabilities of the dual continuous-time Markov chain. Since the update operator \eqref{update operator} can be easily applied to the resulting distribution, Proposition \ref{prop:prop&update} then shows that under these assumptions all filtering distributions are countable mixtures of elementary kernels indexed by the state space of the dual process, with mixture weights proportional to the dual process transition probabilities $p_{\mm,\nn}(t;\theta)$. When these transition probabilities happen to give positive mass only to points $\{\nn\in \M: \ \nn\le \mm\}$, as is the case for a pure-death process, then the right-hand side of \eqref{forward propagation} reduces to a finite sum, and one can construct an exact filter with a computational cost that is polynomial in the number of observations, as shown in \cite{PR14}. 

The above approach can be seen as an alternative to deriving the filtering distribution of the signal by leveraging on a spectral expansion of the transition function $P_{t}$ in \eqref{propagation operator}, which typically requires ad hoc computations and does not lend itself easily to explicit update operations through \eqref{update operator}. Note also that expressions like \eqref{forward propagation} can be used, by taking appropriate limits of $p_{\mm,\nn}(t;\theta)$ as $t\rightarrow 0$, to identify the transition kernel of the signal $P_{t}$ itself, see, e.g., \cite{BEG00,PR14,Gea22}. 


\section{Recursive formulae for filtering and smoothing}\label{sec:algorithms}

In order to translate Proposition \ref{prop:prop&update} into practical recursive formulae for filtering and smoothing, we are going to assume for simplicity of exposition that the time intervals between successive data collections $t_{i}-t_{i-1}$ equal $\Delta$ for all $i$. For ease of reading, we will therefore use the symbol $P_{\Delta}$ instead of $P_{t_{i}-t_{i-1}}$ for the signal transition function over the interval $\Delta=t_{i}-t_{i-1}$. We will also use the established notation whereby $i|0:i-1$ indicates that the argument refers to time $t_{i}=i\Delta$, and we are conditioning on the data collected at times from $0$ to $t_{i-1}=(i-1)\Delta$.

Define 
the  \emph{filtering density}
\begin{equation}\label{filtering distr general}
\nu_{i|0: i}(x_{i}):=p(x_{i}| y_{0:i})\propto \int_{\X^{i}} p(x_{0:i}, y_{0:i}) \d x_{0:i-1}, 
\end{equation}
i.e., the law of the signal at time $t_{i}$ given data up to time $t_{i}$, obtained by integrating out the past trajectory. Define also the \emph{predictive density}
\begin{equation}\label{predictive general}
\nu_{i+1|0:i}(x_{i}):=p(x_{i+1} | y_{0:i}) =\int_{\X} p(x_{i} |  y_{0:i})P_{\Delta}(x_{i+1} |  x_{i})\d x_{i},
\end{equation} 
i.e, the marginal density of the signal at time $t_{i+1}$, given data up to time $t_{i}$. This can be expressed recursively as a function of the previous \emph{filtering density} $p(x_{i} |  y_{0:i})$, as displayed.
Finally, define 
the marginal \emph{smoothing density}
\begin{equation}\label{smoothing distr general}
 \nu_{i|0:n}(x_{i}):=p(x_{i} |  y_{0:n}) \propto \int_{\X^{n}} p(x_{0:n}, y_{0:n}) \d x_{0:i-1}\d x_{i+1:n},
\end{equation}
where the signal is evaluated at time $t_{i}$ conditional on all available data. The first two distributions above are natural objects of inferential interest, whereas the latter is typically used to improve previous estimates once additional data become available. 
Finally, for $\Theta_{\Delta}$ as in Assumption 3 and $t(\cdot,\cdot),T(\cdot,\cdot)$ as in Assumption 2, define for $i=0,\dots,n$ the quantities
\begin{equation}
\begin{aligned}
\label{vartheta and M-sets}
\vartheta_{i|0:i}:=&\,T(y_{i}, \vartheta_{i|0:i-1}),
\quad \quad \ \
\vartheta_{i|0:i-1} := \Theta_{\Delta}(\vartheta_{i-1|0:i-1}),
\quad \quad 
\vartheta_{0|0:-1} := \theta_0.
\end{aligned}
\end{equation} 
Here, $\vartheta_{i|0:i-1}$ denotes the state of the deterministic component of the dual process at time $i$, after the propagation from time $i-1$ and before updating with the datum collected at time $i$, and $\vartheta_{i|0:i}$ the state after such update.

The following Corollary of Proposition \ref{prop:prop&update} extends a result of \cite{PR14} (see also Theorem 1 in \cite{KKK21} for an easier comparison in a similar notation).

\begin{corollary}\label{prop:rec_filtering}
Let Assumptions 1-3 hold, and assume that
\[\nu_{i-1 | 0:i-1} (x) = \sum_{\mm \in \M}w_\mm^{(i-1)}g(x, \mm, \vartheta_{i-1|0:i-1}).\]
Then \eqref{predictive general} can be written, through \eqref{propagation operator}, as 
\begin{equation}\label{prediction in thm}
\begin{aligned}
 \nu_{i|0:i-1}(x) 
=&\,\psi_{\Delta}(\nu_{i-1 | 0:i-1}(x))
 = \sum_{\mm \in \M}w_\mm^{(i-1)'}g(x, \mm, \vartheta_{i|0:i-1}), \\
  w_\mm^{(i-1)'} = &\, \sum_{\nn \in \M}w_{\nn}^{(i-1)}p_{ \nn, \mm}(\Delta; \vartheta_{i-1|0:i-1}), \quad \mm \in \M,
\end{aligned}
\end{equation} 
with $p_{ \nn, \mm}(\Delta; \vartheta_{i|0:i})$ as in \eqref{transition probabilities}. Furthermore, given the observation $y_i$, \eqref{filtering distr general} can be written, through \eqref{update operator}, as
\begin{equation}\label{filtering in thm}
\begin{aligned}
 \nu_{i|0:i}(x)
 = &\, \phi_{y_i}(\nu_{i|0:i-1}(x))
= \sum_{\mm \in \M}w_\mm^{(i)} g(x, \mm, \vartheta_{i|0:i}),\\
w_{\mm}^{(i)} \propto &\,
\mu_{\nn, \vartheta_{i|0:i-1}}(y_{i})w_{\nn}^{(i-1)'}, 
\quad \mm =t(y_{i},\nn),\nn\in \M,
\end{aligned}
\end{equation} 
with $\mu_{\mm,\theta}$ as in \eqref{marginals}.
\end{corollary}

\begin{quote}
\begin{algorithm}[t]
\small
\SetAlgoLined


\KwIn{$Y_{0:n}$, $t_{0:n}$}

\KwResult{$\vartheta_{i|0:i}$, $\MM_{i|0:i}$ and $W_i = \{w _\mm ^{(i)} , \mm \in \MM _{i|0:i} \}$}

\SetKwBlock{Begin}{Initialise}{}

\Begin{

Set $\vartheta_{0|0} = T(Y_0, \theta_0)$ with $T$ as in Assumption 2\\

Set $\MM_{0|0} = \{t(Y_0, \oo)\} = \{\mm^*\}$ and $W_0 = \{1\}$ with $t$ as in  Assumption 2\\

Compute $\vartheta_{1|0}$ from $\vartheta_{0|0}$ as in \eqref{vartheta and M-sets}\\

Set $\MM^* = \B(\MM_{0|0})$ and $W^* = \{p_{\mm^*, \nn}(\Delta, \vartheta_{0|0}), \nn \in \MM^*\}$ with  $p_{\mm, \nn}$ as in \eqref{transition probabilities}

}

\For{$i$ from $1$ to $n$}{

\SetKwBlock{Begin}{Update}{}

\Begin{

Set $\vartheta_{i|0:i} = T(Y_i, \vartheta_{i|0:i-1})$\\

Set $W_{i} = \left\{\frac{w_\mm^* \mu_{\mm, \vartheta_{i|0:i-1}}(Y_i)}{\sum_{\nn \in \MM^*}w_\nn^* \mu_{\nn, \vartheta_{i}}(Y_i)}, \mm \in \MM^*\right\}$ with $\mu_{\mm, \theta}$ defined as in \eqref{marginals}

Set $\MM_{i|0:i} = \{t(Y_i, \mm), \mm \in \MM^*\}$ and update the labels in $W_i$\\
Copy $\vartheta_{i|0:i}$, $\MM_{i|0:i}$ and $W_{i}$ to be reported as the output
}

%
%
%

\SetKwBlock{Begin}{Propagation}{}

\Begin{

Compute $\vartheta_{i+1|0:i}$ from $\vartheta_{i|0:i}$\\

Set $\MM^* = \B(\MM_{i|0:i})$ and $W^* = \bigg\{\displaystyle{\sum_{\mm \in \MM_{i|0:i}}}w_{\mm}^{(i)} p_{\mm, \nn}(\Delta,\vartheta_{i|0:i}), \nn \in \MM^*\bigg\}$}}
\begin{quote}
\caption{\small Filtering}\label{algorithm}
\textbf{Note:} $\MM_{i|0:i}=\{\mm \in \mathcal{M}:\ w_\mm ^{(i)}>0\}\subset \mathcal{M}$ is the support of the weights of $\nu _{i|0:i}$;
$\B(\mm)$ denotes the states reached by the dual process from $\mm$, and $\B(\MM)$ those reached from all $\mm \in \MM$.
\end{quote}
\end{algorithm}
\end{quote}

\vspace{-3mm}
Algorithm \ref{algorithm} outlines the pseudo-code for implementing the update and propagation steps of Corollary \ref{prop:rec_filtering}.
How to use these results efficiently can depend on the model at hand. When the transition probabilities $p_{\mm,\nn}(t;\theta)$ are available in closed form, their use could lead to the best performance, but can also at times face numerical instability issues (as is the case pointed out in Section \ref{sec:CIR} below).
When the transition probabilities $p_{\mm,\nn}(t;\theta)$ are not available in closed form, one can approximate them by simulating $N$ replicates of the dual component $M_{t}$, and then regroup probability masses according to the arrival states as done in \eqref{prediction in thm}.
In our framework, the dual process is typically easier to simulate than the original process, given its discrete state space. For instance, pure-death or B\&D processes are easily simulated using a Gillespie algorithm \citep{G07}, whereby one alternates sampling waiting times and jump transitions for the embedded chain. Depending on the dual process, there might also be more efficient simulation strategies.

A different type of approximation of the propagation step \eqref{prediction in thm} in Corollary \ref{prop:rec_filtering} can be based on pruning the transition probabilities or the arrival weights under a given threshold, followed by a renormalisation of the weights.
Both this approximation strategy and that outlined above assign positive weights only to a finite subset of $\M$, hence they overcome the infinite dimensionality of the dual process state space. 
\cite{KKK21} showed that the latter strategy allows to control the approximation error while retaining almost entirely the distributional information, thus affecting the inference negligibly. 
In the next sections we will investigate such strategies for two hidden Markov models driven by Cox--Ingersoll--Ross  and $K$-dimensional Wright--Fisher diffusions.

Now, in order to describe the marginal smoothing densities \eqref{smoothing distr general}, we need an additional assumption and some further notation.

\begin{quote}
\textbf{Assumption 4} For $h$ as in Assumption 3, there exist functions $d:\M^{2}\to\M$ and $e:\Theta^{2}\to\Theta$ such that for all $x\in\X,\ \mm,\mm' \in\M,\ \theta, \theta' \in \Theta$
\begin{equation}\label{hh_stab}
 h(x, \mm, \theta)h(x, \mm', \theta') 
 =C_{\mm, \mm', \theta, \theta'} h(x, d(\mm, \mm'), e(\theta, \theta')),
\end{equation}
where $C_{\mm, \mm', \theta, \theta'}$ is constant in $x$.
\end{quote}

Denote by $\avt_{i}, \avt_{i}'$ the quantities defined in 
\eqref{vartheta and M-sets} computed backwards. Equivalently, these are computed as in \eqref{vartheta and M-sets} with data in reverse order, i.e.~using $y_{n:0}$ in place of $y_{0:n}$, namely
\begin{equation}\label{backward lambda e vartheta}
\begin{aligned}
\avt_{i|i+1:T}=&\,\Theta_{\Delta}(\avt_{i+1|i+1:T}), \quad \quad 
\avt_{i|i:T}=T(y_{i},\avt_{i|i+1:T}),\quad \quad 
\avt_{T|T}=T(y_{T},\theta_{0})
\end{aligned}
\end{equation} 

The following result extends Proposition 3 and Theorem 4 of \cite{KKK21}:
\begin{proposition}\label{prop: smoothing}
Let Assumptions 1-4 hold, and let $\nu _0= \pi$. Then, for $0\leq i\leq n-1$, we have
\begin{equation}
 p(x_{i} |  y_{0:n})=
\sum_{\mm \in \M,\ \nn \in \M}
w_{\mm,\nn}^{(i)}g(x_{i}, d(\mm,\nn), e(\avt_{i|i+1:n},\vartheta_{i|0:i})),
\end{equation}
 with 
\begin{equation}
\begin{aligned}
w_{\mm,\nn}^{(i)}
\propto&\, \overleftarrow w^{(i+1)}_{\mm} w_\nn^{(i)}C_{\mm, \nn, \avt_{i|i+1:n},\vartheta_{i|0:i}},\\
\aw _\mm ^{(i+1)}=&\,\sum _{\nn \in \M} \aw _{\nn}^{(i+2)}\mu _{\nn, \avt _{i+1|i+2:n}}(y_{i+1})p_{t(y_{i+1}, \nn),\mm }(\Delta ; \avt _{i+1|i+1:n})
\end{aligned}
\end{equation} 
 $w_\nn^{(i)}$ as in \eqref{filtering in thm} and $C_{\mm, \nn, \avt_{i|i+1:n},\vartheta_{i|0:i}}$ as in \eqref{hh_stab}.
\end{proposition}
\begin{proof}
Note that Bayes' Theorem and conditional independence allow to write  \eqref{smoothing distr general} as
\begin{equation}\label{smoothing_expression}
\nu _{i|0:n}(x_i) = p(x_i| y_{0:n}) 
\propto p(y_{i+1:n}| x_i)
\nu_{i|0:i}(x_i)
\end{equation} 
where the right-hand side involves the filtering distribution, available from Corollary \ref{prop:rec_filtering}, and the  so called \emph{cost-to-go function} $p(y_{i+1:n}| x_i)$ (sometimes called information filter), which is the likelihood of future observations given the current signal state. Along the same lines as Proposition 3 in \cite{KKK21} we find that
\begin{equation}
p(y_{i+1:n} |  x_i) = \sum _{\mm \in \M}\aw _\mm ^{(i+1)}h(x_i, \mm , \avt _{i|i+1:n})
\end{equation}
with $\aw _\mm ^{(i+1)}$ as in the statement. The main claim can now be proved along the same lines as Theorem 4 in \cite{KKK21}. 
\end{proof}

The main difference between the above result and Theorem 4 in \cite{KKK21} lies in the fact that the support of the weights $\{\aw _\mm ^{(i+1)}, \mm \in \M\}$ (which in \cite{KKK21} is denoted by $\aMM_{i|i+1:n}$) can be countably infinite and coincide with the whole of $\M$. Indeed, which points of $\M$ have positive weight are determined by the transition probabilities of the dual process, which in the present framework is no longer assumed to make only downward moves in $\M$. Section \ref{sec:experiments} will deal with this possibly infinite support for a concrete implementation of the inferential strategy.


\section{Cox--Ingersoll--Ross hidden Markov models}\label{sec:CIR}

The Cox--Ingersoll--Ross diffusion, also known as the square-root process,
is widely used in financial mathematics for modelling short-term interest rates and stochastic volatility, see \cite{CIR85,CS92,F94,H93,GY03}. It also belongs to the class of continuous-state branching processes with immigration, arising as the large-population scaling limit of certain branching Markov chains \citep{KW71} and as the time-evolving total mass of a Dawson--Watanabe branching measure-valued diffusion \citep{EG93b}.
		
Let $X_{t}$ be a CIR diffusion on $\R_{+}$ that solves the one-dimensional SDE
\begin{equation}\label{SDE}
    \d X_t = \l( \delta \sigma^2 - 2 \gamma X_t \r)dt + 2\sigma \sqrt{X_t} d B_t,\quad \quad X_{0}\ge0,
\end{equation}
where $\delta,\gamma,\sigma>0$, which is reversible with respect to the Gamma density $\pi=\text{Ga}(\delta/2, \gamma / \sigma^2)$.
The following proposition identifies a B\&D  process as dual to the CIR diffusion.

\begin{proposition}\label{prop: CIR duality}
Let $X_{t}$ be as in \eqref{SDE}, let $M_{t}$ be a \emph{B\&D} process on $\Z_{+}$ which jumps from $m$ to $m+1$ at rate $\lambda_m = 2\sigma^2 (\delta/2 + m)(\theta - \gamma /\sigma^2)$  and to $m-1$ at rate $\mu_m = 2 \sigma^2 \theta m$, and let 
\begin{equation} \label{CIR duality function}
    h(x, m, \theta) = \frac{\Gamma(\delta/2)}{\Gamma(\delta/2+m)} \Big(\frac{\gamma}{\sigma^2}\Big)^{-\delta/2}\theta^{\delta/2 +m}x^m e^{-(\theta - \gamma/\sigma^2)x}.
\end{equation}
Then \eqref{duality identity} holds with $D_{t}=M_{t}$. 
\end{proposition}
\begin{proof}
The infinitesimal generator associated to \eqref{SDE} is
\begin{equation}\nonumber
\A f(x)=( \delta \sigma^2 - 2 \gamma x )f'(x)+ 2\sigma^{2} x f''(x),
\end{equation} 
for $f:\R_{+}\rightarrow \R$ vanishing at infinity.
Letting $h(x,m) = h(x,m, \theta)$ denote \eqref{CIR duality function} omitting the dependence on $\theta$ to make notations lighter, a direct computation yields 
\begin{align}
    \mathcal{A} h(\cdot, m)(x) 
    =&\,(\delta \sigma^2 - 2 \gamma x)\Big(mx^{m-1}-x^{m}(\theta-\gamma/\sigma^{2})\Big)\frac{\Gamma(\delta/2)}{\Gamma(\delta/2+m)} \Big(\frac{\gamma}{\sigma^2}\Big)^{-\delta/2}\theta^{\delta/2 +m} e^{-(\theta - \gamma/\sigma^2)x}
\\
&\,    + 2\sigma^{2} x
\Big(m(m-1)x^{m-2}+x^{m}(\theta-\gamma/\sigma^{2})^{2}-2mx^{m-1}(\theta-\gamma/\sigma^{2})\Big) \\&\, 
\times\frac{\Gamma(\delta/2)}{\Gamma(\delta/2+m)}\Big(\frac{\gamma}{\sigma^2}\Big)^{-\delta/2}\theta^{\delta/2 +m} e^{-(\theta - \gamma/\sigma^2)x}\\
=&\,
\frac{\delta\sigma^{2}m\theta}{\delta/2+m-1}h(x,m-1)
+2\gamma(\theta-\gamma/\sigma^{2})\frac{\delta/2+m}{\theta}h(x,m+1)\\
&\,-[2\gamma m+\delta\sigma^{2}(\theta-\gamma/\sigma^{2})] h(x,m)
+2\sigma^{2}m(m-1)\frac{\theta}{\delta/2+m-1}h(x,m-1)\\
&\,+2\sigma^{2}(\theta-\gamma/\sigma^{2})^{2}\frac{\delta/2+m}{\theta}h(x,m+1)
    -4\sigma^{2}m(\theta-\gamma/\sigma^{2})h(x,m)
\\
    =&\, 2 \sigma^2 \theta m h(x, m-1) + 2\sigma^2 (\delta/2 + m)(\theta - \gamma /\sigma^2) h(x, m+1)  \notag\\
&\,    - [2\gamma m+\sigma^{2}(\delta+4m)(\theta-\gamma/\sigma^{2})]h(x, m).
\end{align}
where it can be checked that
\begin{align}
2\sigma^{2}\theta m+2\sigma^{2}(\delta/2+m)(\theta-\gamma/\sigma^{2})
=2\gamma m+\sigma^{2}(\delta+4m)(\theta-\gamma/\sigma^{2}).
\end{align}
Hence the r.h.s.~equals
\begin{equation}\nonumber
\B g(m)=\lambda_{m}[g(m+1)-g(m)]+\mu_{m}[g(m-1)-g(m)]
\end{equation} 
with $g(\cdot):=h(x,\cdot)$, $\lambda_{m}=2\sigma^2 (\delta/2 + m)(\theta - \gamma /\sigma^2)$, and $\mu_{m}=2 \sigma^2 \theta m$, which  is the infinitesimal generator of a B\&D  process with rates $\lambda_{m},\mu_{m}$. 
The claim now follows from Proposition 1.2 in \cite{JK14}.
\end{proof}


Assign now prior $\nu_{0}=\text{Ga}(\delta/2,\gamma/\sigma^{2})$ to $X_{0}$, and assume Poisson observations are collected at equally spaced intervals of length $\Delta$. Specifically, $Y|X_{t}=x \overset{\text{iid}}{\sim} \text{Po}(\tau x)$, for some $\tau>0$. 
By the well-known conjugacy to Gamma priors, we have that $X_{t}|Y=y\sim \text{Ga}(\delta/2+y, \gamma / \sigma^2+\tau)$. Without loss of generality, we can set $\tau=1$, which allows to interpret the update of the gamma rate parameter as the size of the conditioning data set.
The filtering algorithm starts by first updating the prior $\nu_{0}$ to $\nu_{0|0}:=\phi_{Y_{0}}(\nu_{0})$. If we observe $Y_{0}=(Y_{0,1},\ldots,Y_{0,k})$ at time $0$, then $\nu_{0|0}$ is the law of $X_{0}|\sum_{j=1}^{k}y_{0,j}=m\sim \text{Ga}(\delta/2+m, \gamma / \sigma^2+k)$. Then $\nu_{0|0}$ is  propagated forward for a $\Delta$ time interval, yielding $\nu_{1|0}:=\psi_{\Delta}(\nu_{0|0})$.
In light of Proposition \ref{prop: CIR duality}, an application of \eqref{forward propagation} to  $\nu_{0|0}$ yields the infinite mixture
\begin{equation}\label{CIR propagation}
\psi_\Delta \left( \text{Ga}(\delta/2+m, \gamma / \sigma^2+k)\right) = \sum_{n\ge0} p_{m, n}(\Delta) \text{Ga}(\delta/2+n, \gamma / \sigma^2+k),
\end{equation} 
where $p_{m, n}(t)$ are the transition probabilities of $M_{t}$ in Proposition \ref{prop: CIR duality}.  Hence, the law of the signal is indexed by $\Z_{+}$, the state space of the dual process. While after the update at time 0 mass one is assigned to the sum of the observations $\sum_{j=1}^{k}y_{0,j}=m$, after the propagation the mass is spread over the whole $\Z_{+}$ by the effect of the dual process. 
We then observe $Y_{1}\sim f_{X_{1}}$ at time 1, which is used to update $\nu_{1|0}$ to $\nu_{1|1}$ and has the effect of shifting the probability masses of the mixture weights. For example, the weight $p_{m,n}(\Delta)$ in \eqref{CIR propagation} is assigned to $n\in \Z_{+}$, but after the update based on $Y_{1}=(Y_{1,1},\ldots,Y_{1,k'})$  it will be assigned to $n+m'$ if $\sum_{j=1}^{k'}y_{1,j}=m'$, on top of being transformed according to \eqref{mixture update}. We then propagate forward again and proceed analogously.

When the current distribution of the signal, after the update, is given by a mixture of type $\sum_{m\in \Z_{+}}w_{m}\text{Ga}(\delta/2+m, \gamma / \sigma^2+k)$, it is enough to rearrange the mixture weights after the propagation step as done in \eqref{weights rearrangement after propagation}. 
%

The main difference with qualitatively similar equations found in \cite{KKK21} is now given by the transition probabilities $p_{m, n}(t)$ in \eqref{CIR propagation}, which are those of the B\&D process in Proposition \ref{prop: CIR duality}. Before tackling the problem of how to use the above expressions for inference, we try to provide further intuition of the extent and implications of such differences.
To this end, consider the simplified parameterization $\alpha=\delta/2, \beta=\gamma/\sigma^{2},  \sigma^{2}=1/2, \tau=1$, whereby 
one can check that the embedded chain of the B\&D  process of Proposition \ref{prop: CIR duality} has jump probabilities 
\begin{equation}\nonumber
\begin{aligned}
p_{m,m+1}=\frac{k(\alpha+m)}{k(\alpha+m)+m(\beta+k)},\quad \quad 
p_{m,m-1}=1-p_{m,m+1}.
\end{aligned}
\end{equation} 
Here $m,k$ are the same as in the left-hand side of \eqref{CIR propagation}, so $m/k$ is the sample mean.
It is easily verified that $p_{m, m+1}<p_{m, m-1}$ if $m/k>\alpha/\beta$ and viceversa.
Therefore, the dual evolves on $\Z_{+}$ so that it reverts $m/k$ to the prior mean $\alpha/\beta$. 
Indeed, the dual has Negative Binomial ergodic distribution $\text{NBin}\left(\alpha,\beta/(\beta+k)\right)$, 
whose mean is $k\alpha/\beta$, i.e., such that $m/k$ on average coincides with $\alpha/\beta$.

Recall now that the dual process elicited in \cite{PR14} for the CIR model is $D_{t}=(M_{t},\Theta_{t})$, with $M_{t}$ a pure-death process with rates from $m$ to $m-1$ equal to $2\sigma^{2}\theta$ and $\Theta_{t}$ a deterministic process that solves $\d \Theta_{t}/\d t=-2\sigma^{2}\Theta_{t}(\Theta_{t}-\gamma/\sigma^{2})$, $\Theta_{0}=\theta$. This dual has a single ergodic state given by $(0,\beta)$ (note that \cite{PR14} uses a slightly different parameterization, where the ergodic state $(0,\beta)$ means that, in the limit for $t\rightarrow \infty$, the gamma parameters are the prior parameters). 
In particular, as $t\rightarrow \infty$, this entails the convergence of $p_{m,n}(t)$ in \eqref{CIR propagation} to 1 if $n=0$ and 0 elsewhere. Whence the strong ergodic convergence $\psi_{t}(g(x,m,\theta))\rightarrow \pi$ as $t\rightarrow \infty$, whereby the effect of the observed data become progressively negligible as $t$ increases.
One could then argue that in the long run, the filtering strategy based on the pure-death dual process in \cite{PR14} completely forgets the collected data. As a consequence, one could expect filtering with long-spaced observations (relative to the forward process autocorrelation) to be similar to using independent priors at each data collection point. On the other hand, the B\&D dual can be thought as not forgetting but rather spreading around the probability mass in such a way as to preserve convergence of the empirical mean to the prior mean. It is not obvious a priori which of these two scenarios could be more beneficial in terms of filtering, hence in Section \ref{sec: CIR experiments} we provide numerical experiments for comparing the performance of strategies based on these different duals.
%

In view of such experiments, note that the transition probabilities of the above B\&D dual are in principle available in closed form (cf.~\cite{B64,CS12}), but their computation is  prone to numerical instability.  
Alternatively, we can approximate the transition probabilities $p_{m,n}(t)$ in \eqref{CIR propagation} by drawing $N$ sample paths of the dual started in $m$ and use the empirical distribution of the arrival points.
This can in principle be done through the Gillespie algorithm \citep{G07}, which alternates sampling waiting times and jumps of the embedded chain. A faster strategy can be achieved by writing the B\&D  rates in Proposition \ref{prop: CIR duality} as $\lambda_{m}=\lambda m+\beta$ and $\mu_{m}=\mu m$ with
\begin{equation}\label{algrates}
\begin{aligned}
\lambda =2\sigma^2 (\theta -\gamma/\sigma^2),\qquad \beta=\sigma^2\delta(\theta -\gamma/\sigma^2),\qquad 
\mu = 2\sigma^2\theta,
\end{aligned}
\end{equation}
where $\lambda,\mu$ represent the per capita birth and death rate and $\beta$ is the immigration rate. Then write $M_t = A_t + B_t$ where $A_t$ is the population size of the descendant of autochthonous individuals (already in the population at $t=0$), and $B_t$ the descendants of the immigrants. 
These rates define a linear B\&D process, whereby \cite{T18} suggests simulating $A_t$ by drawing, given $A_0=i$,
\begin{equation}\label{tav18}
    F\sim \text{Bin}(i, g(t)), \qquad A_t \sim \text{NBin}(F, h(t))+F,
\end{equation}
with $h(t)=(\lambda-\mu)/(\lambda \exp\{(\lambda-\mu)t\}-\mu)$ and $g(t)=h(t)\exp\{(\lambda-\mu)t\}$, 
with the convention NBin$(0,p)= \delta_0$.
Let now $N_s$ be the number of immigrants up to time $s$, which follows a simple Poisson process with rate $\beta$, so given $N_{t}$ the arrival times are uniformly distributed on $[0,t]$.
Once in the population, the lineage of each immigrating individual follows again a B\&D  process and can be simulated using \eqref{tav18} starting at $i=1$. Summing the numerosity of each immigrant family at time $t$ yields $B_t$.



\section{Wright--Fisher hidden Markov models}\label{sec:WF}

The $K$-dimensional WF diffusion is a widely studied classical model in population genetics (see \citep{EK86,EG93a,E09,H07} and references therein), recently used also in a statistical framework \citep{PR14,KKK21}. See also \cite{Cea15} for connections with statistical physics.
 It takes values in the simplex 
\begin{equation}\nonumber
\Delta_{K}=\bigg\{\xx\in[0,1]^{K}: \sum_{1\le i\le K}x_{i}=1\bigg\}
\end{equation} 
and, in the population genetics interpretation, it models the temporal evolution of $K$ proportions of types in an underlying large population. Its infinitesimal generator on $C^{2}(\Delta_{K})$ is
\begin{equation}\label{WF generator}
\AA =  \frac{1}{2}\sum_{i,j=1}^{K}x_{i}(\delta_{ij}-x_{j})\frac{\partial^{2}}{\partial x_{i}\partial x_{j} }
+\frac{1}{2}\sum_{i=1}^{K}(\alpha_i - \theta x_i)\frac{\partial}{\partial x_{i}}
\end{equation} 
for $\aalpha=(\alpha_{1},\ldots,\alpha_{K})\in \R_{+}^{K}$, $\theta=\sum_{i=1}^{K}\alpha_{i}$, and its reversible measure is the Dirichlet distribution whose density with respect to Lebesgue measure is 
\begin{equation}\label{Dirichlet distribution}
\pi_{\aalpha}(\xx)=\frac{\Gamma(\theta)}{\prod_{i=1}^{K}\Gamma(\alpha_{i})}x_{1}^{\alpha_{1}-1}\cdots x_{K}^{\alpha_{K}-1}, \quad x_{K}=1-\sum_{i=1}^{K-1}x_{i}.
\end{equation} 
See for example \cite{EK86}, Chapter 10. The transition density of this model is (cf., e.g., \cite{EG93a}, eqn.~(1.27))
\begin{equation}\label{WF transition}
p_{t}(\xx,\xx')=\sum_{m=0}^{\infty}d_{m}(t)
\sum_{\mm\in \Z_{+}^{K}:|\mm|=m}\text{MN}(\mm; m,\xx)\pi_{\aalpha+\mm}(\xx'),
\end{equation} 
where $\text{MN}(\mm; m,\xx)=\binom{m}{m_{1},\ldots,m_{K}}\prod_{i=1}^{K}x_{i}^{m_{i}}$, and where $d_{m}(t)$ are the transition probabilities of the block counting process of Kingman's coalescent on $\Z_{+}$, which has an entrance boundary at $\infty$. Cf., e.g., \cite{EG93a}, eqn.~(1.12).

It is well known that a version of Kingman's typed coalescent with mutation is dual to the WF diffusion. This can be seen as a death process on $Z_{+}^{K}$ which jumps from $\mm$ to $\mm-\ee_{i}$ at rate
\begin{equation}\label{kingman's rate}
q_{\mm,\mm-\ee_{i}}=m_{i}(\theta+|\mm|-1)/2.
\end{equation} 
Here $\ee_{i}$ is the canonical vector in the $i$-th direction.
See, for example, \cite{EG09,GS09,Eea10}; see also \cite{PR14}, Section 3.3. The above death process with transitions $d_{m}(t)$ is indeed the process that counts the surviving blocks of the typed version without keeping track of which types have been removed. 

Recall now that a Moran model with $N$ individuals of $K$ types is a particle process with overlapping generations whereby at discrete times a uniformly chosen individual is removed and another, uniformly chosen from the remaining individuals, produces one offspring of its own type, leaving the total population size constant. See, e.g., \cite{E09}. In the presence of mutation, upon reproduction, the offspring can mutate to type $j$ at parent-independent rate $\alpha_{j}$. The generator of such process on the set $B(\Z_{+}^{K})$ of bounded functions on $\Z_{+}^{K}$ can be written in terms of the multiplicities of types $\nn\in \Z_{+}^{K}$ as 
\begin{equation}\label{moran generator}
\B f(\nn)
=\frac{1}{2}\sum_{1\le i\ne j\le K}n_{i}(\alpha_{j}+n_{j})f(\nn-\ee_{i}+\ee_{j})
-\frac{1}{2}\sum_{1\le i\ne j\le K}n_{i}(\alpha_{j}+n_{j})f(\nn),
\end{equation} 
where an individual of type $i$ is removed at rate $n_{i}$, the number of individuals of type $i$, and is replaced by an individual of type $j$ at rate $\alpha_{j}+n_{j}$. 

The following proposition extends a result in \cite{Cea15} (cf.~Section 5) and shows that the above Moran model is dual to the WF diffusion with generator \eqref{WF generator}.

\begin{proposition}\label{prop: WF duality}
Let $X_{t}$ have generator \eqref{WF generator}, let $N_{t}\in \Z_{+}^{K}$ be a Moran model which from $\nn$ jumps to $\nn-\ee_{i}+\ee_{j}$ at rate $n_i(\alpha_j + n_j)/2$, and let 
\begin{equation}\label{eq: duality function Moran-WF}
h(\xx,\nn)=\frac{\Gamma(\theta+|\nn|)}{\Gamma(\theta)}\prod_{i=1} ^K\frac{\Gamma\left(\alpha_i\right)}{\Gamma\left(\alpha_i+n_{i}\right)}x_{i}^{n_{i}}, \quad \quad 
\theta= \sum _{i=1}^K \alpha _i.
\end{equation}
Then \eqref{duality identity} holds with $D_{t}=N_{t}$ and $h$ as above. 
\end{proposition}

\begin{proof}
From \eqref{WF generator}, since $\theta=\sum_{i=1}^{K}\alpha_{i}$, we can write
\begin{align}
2\AA 
=&\,  \sum_{1\le i\le K}x_{i}(1-x_{i})\frac{\partial^{2}}{\partial x_{i}^{2}}-\sum_{1\le i\ne j\le K}x_{i}x_{j}\frac{\partial^{2}}{\partial x_{i}\partial x_{j}} 
+\sum_{1\le i\le K}(\alpha_i(1- x_{i})-x_{i}\sum_{1\le j\le K,j\ne i}\alpha_{j})\frac{\partial}{\partial x_{i}}\\
=&\,  \sum_{1\le i\ne j\le K}x_{i}x_{j}\frac{\partial^{2}}{\partial x_{i}^{2}}
-\sum_{1\le i\ne j\le K}x_{i}x_{j}\frac{\partial^{2}}{\partial x_{i}\partial x_{j}}
+\sum_{1\le i\le K}\alpha_i\sum_{1\le j\le K,j\ne i}x_{j}\frac{\partial }{\partial x_{i}}
-\sum_{1\le i\ne j\le K}\alpha_{j}\frac{\partial }{\partial x_{i}}.
\end{align}
Then one can check that
\begin{align}
2\AA h(\xx,\nn)
=&\,\sum_{1\le i\ne j\le K}n_{i}(\alpha_{i}+n_{i}-1)\frac{\Gamma(\theta+|\nn|)}{\Gamma(\theta)}\xx^{\nn-\ee_{i}+\ee_{j}}\prod_{h=1} ^K\frac{\Gamma\left(\alpha_h\right)}{\Gamma\left(\alpha_h+n_{h}\right)}\\
&\,-\sum_{1\le i\ne j\le K}n_{i}(\alpha_{j}+n_{j})\xx^{\nn}\frac{\Gamma(\theta+|\nn|)}{\Gamma(\theta)}
\prod_{h=1} ^K\frac{\Gamma\left(\alpha_h\right)}{\Gamma\left(\alpha_h+n_{h}\right)}\\
=&\,\sum_{1\le i\ne j\le K}n_{i}(\alpha_{j}+n_{j})h(\xx,\nn-\ee_{i}+\ee_{j})
-\sum_{1\le i\ne j\le K}n_{i}(\alpha_{j}+n_{j})h(\xx,\nn).
\end{align}
Hence we have
\[(\AA h (\cdot,\nn))(\xx) = (\B h(\xx,\cdot) ) (\nn)\]
where the right hand side  is \eqref{moran generator} applied to $h(\xx,\nn)$ as a function of $\nn$. The claim now follows from Proposition 1.2 in \cite{JK14}.
\end{proof}

Assign now prior $\nu_{0}=\pi_{\aalpha}$ to $X_{0}$, and assume categorical observations so that  $\P(Y=j|X_{t}=x)=x_{j}$. By the well-known conjugacy to Dirichlet priors, we have $X_{t}|Y=y\sim \pi_{\aalpha+\delta_{y}}$, where $\aalpha+\delta_{y}=(\alpha_{1},\ldots,\alpha_{j}+1,\ldots,\alpha_{K})$ if $y=j$. When multiple categorical observations with vector of multiplicities $\mm\in \Z_{+}^{K}$ are collected, we write $\pi_{\aalpha+\mm}$. The filtering algorithm then proceeds by first updating $\nu_{0}$ to $\nu_{0|0}:=\phi_{Y_{0}}(\nu_{0})=\pi_{\aalpha+\mm}$, if $Y_{0}=(Y_{0,1},\ldots,Y_{0,k})$ yields multiplicities $\mm$, then propagating $\nu_{0|0}$ to $\nu_{1|0}:=\psi_{\Delta}(\nu_{0|0})$.
In light of the previous result, an application of \eqref{forward propagation} to $\nu_{0|0}=\pi_{\aalpha+\mm}$ yields the mixture
\begin{equation}\label{WF propagation}
    \psi_{\Delta} \left(\pi_{\aalpha+\mm} \right) = \sum_{\nn: |\nn|=|\mm|} p_{\mm,\nn}(\Delta) \pi_{\aalpha+\nn},
\end{equation}
where $p_{\mm,\nn}(\Delta)$ are the transition probabilities of $N_{t}$ in Proposition \ref{prop: WF duality} over the interval $\Delta$. We then observe $Y_{1}|X_{1}$, which is in turn used to update $\nu_{1|0}$ to $\nu_{1|1}$, as so forth. We refer again the reader to \cite{KKK21}, Section 2.4.2, for details on qualitatively similar recursive formulae.

In \eqref{WF propagation}, the overall multiplicity $|\nn|$ equals the original $|\mm|$, as an effect of the population size preservation provided by the Moran model. The space $\{\nn: |\nn|=|\mm|\}$ is finite, which shows that Assumption 3 need not require the presence of a death-like process to have filtering distributions being finite mixtures. However, it is not obvious \emph{a priori} how \eqref{WF propagation} compares in terms of practical implementation with the different representation obtained in \cite{PR14}, namely 
\begin{equation}\label{PR14 WF propagation}
    \psi_{\Delta} \left(\pi_{\aalpha+\mm} \right) = \sum_{\nn: |\nn|\le |\mm|} \hat p_{\mm,\nn}(\Delta) \pi_{\aalpha+\nn}
\end{equation} 
where $\hat p_{\mm,\nn}(\Delta)$ are the transition probabilities of the death process on $\Z_{+}^{K}$ with rates \eqref{kingman's rate}. Similarly to what has already been discussed for the CIR case, the death process dual has a single ergodic state given by the origin $(0,\ldots,0)$, which entails the convergence of $\hat p_{\mm,\nn}(t)$ to $1$ if $\nn=(0,\ldots,0)$ and 0 elsewhere, implying the strong convergence $    \psi_t \left(\pi_{\aalpha+\mm} \right)\rightarrow \pi_{\aalpha}$ in \eqref{PR14 WF propagation}. This is ultimately determined by the fact that Kingman's coalescent removes lineages by coalescence and mutation until absorption to the empty set.

At first glance, a similar convergence is seemingly precluded  to \eqref{WF propagation}. However, we note in the first sum of \eqref{moran generator} that the new particle's type is either resampled from the survived particles or drawn from the baseline distribution, in which case the new particle is of type $j$ with (parent-independent) probability $\alpha_{j}/\sum_{i=1}^{K}\alpha_{i}$. Hence each particle will be resampled from the baseline distribution in finite time. 
Together with the fact that $\sum_{j=1}^{K} \pi_{\aalpha+\delta_{j}}\alpha_{j}/\sum_{i=1}^{K}\alpha_{i}=\pi_{\aalpha}$, which follows from Proposition G.9 in \cite{bGvdV}, and considering that the number of particles is finite, we can therefore expect that, as $t\rightarrow \infty$, we have the convergence $\psi_t \left(\pi_{\aalpha+\mm} \right)\rightarrow \pi_{\aalpha}$ in \eqref{WF propagation} also for this case.


The transition probabilities  $p_{\mm,\nn}(t)$  in \eqref{WF propagation}, induced by the Moran model, are not available in closed form. This poses a limit on the direct applicability of the presented algorithms for numerical experiments. The first alternative is then to approximate them by drawing $N$ points from the discrete distribution on the dual space before the propagation, making use of the Gillespie algorithm to draw as many paths, and evaluating the empirical distribution of the arrival points. Alternative approximations are suggested by the fact that an appropriately rescaled version of the Moran model converges in distribution to a WF diffusion (see, e.g., \cite{E09}, Lemma 2.39). Indeed, a spatial rescaling of the Moran model in \eqref{moran generator} to get proportions in place of multiplicities of types results in the generator 
\begin{equation}
\C_{|\nn|} f(\xx)
=\sum_{1\le i\ne j\le K}\frac{n_{i}}{|\nn|}\frac{\alpha_{j}+n_{j}}{|\nn|}\bigg[f\bigg(\xx-\frac{\ee_{i}}{|\nn|}+\frac{\ee_{j}}{|\nn|}\bigg)-f(\xx)\bigg],
\end{equation} 
where $x_{i}:=n_{i}/|\nn|$. A classical argument based on a Taylor expansion for $f$ now leads to write $|\nn|^{2}\C_{|\nn|} f=\AA f+O(|\nn|^{-1})$, with $\AA$ as in \eqref{WF generator} and where $O(|\nn|^{-1})$ represent a remainder term which goes to zero with $|\nn|^{-1}$. The claim then be based on classical arguments following, e.g., \cite{EK86}, Theorem 4.8.7.
We could therefore use a WF diffusion to approximate the Moran dual transitions in \eqref{WF propagation}. Since the spatially rescaled Moran model takes values $\mm/|\nn|$ with $\mm\in \M$ such that $0\le |\mm|\le |\nn|$, to the above end it suffices to discretize the states of the WF diffusion through binning, e.g., given a state $\xx$ of the approximating WF diffusion, we take as state of the Moran model the point $[|\nn|\xx]:=([|\nn|x_{1}],\ldots,[|\nn|x_{K}])$, where $[|\nn|x_{i}]$ is the approximation of $|\nn|x_{i}$ to the closest integer in $\{0,\ldots,|\nn|\}$. The functionals of interest can thus be evaluated through the same procedure which uses the original Moran model,  i.e., through \eqref{WF propagation}, based on the WF diffusion transition probabilities.
This strategy  in principle has the drawback of having to deal with the intractable terms $d_{m}(t)$ in the transition function expansion \eqref{WF transition} of the diffusion, hurdle overcome by adopting the solution proposed by \cite{JS17}.

It is also known that one could also construct a sequence of WF discrete Markov chains with non-overlapping generations indexed by the population size which, upon appropriate rescaling, converge weakly to the desired WF diffusion (see, e.g., \cite{KT81}, Sec.~15.2.F or \cite{E09}, Sec~4.1). Since two sequences that converge to the same limit can to some extent be considered close to each other, one could then consider a WF discrete chain indexed by $|\nn|$ with a parameterization that would make it converge to \eqref{WF generator}, and use it to approximate the Moran transition probabilities. This would permit a straightforward implementation, given WF discrete chains have multinomial transitions. In Section \ref{sec: WF experiments} we compare the performance of all the above mentioned strategies.

\section{Numerical experiments}\label{sec:experiments}

To illustrate how the above results can be used in practice and how they perform in comparison with other methods, we are going to consider particle approximations of the dual processes for evaluating their transition probabilities, which in turn are used in \eqref{weights rearrangement after propagation} in place of the true transition probabilities to evaluate the predictive distributions for the signal, denoted here $\hat p(x_{k+1}|y_{1:k})$. We compare these distributions with the exact predictive distribution obtained through the results in \cite{PR14} and those obtained through \emph{bootstrap particle filtering} which make use of the signal transition function. Particle filtering can be considered the state of the art for this type of inferential problems, a general reference being \cite{CP20}. The notable difference between these two approaches is that bootstrap particle filtering operates on the original state space of the signal, whereas filtering based on dual processes index the filtering mixtures using the dual state space, which in the present framework is discrete.

We first briefly describe the specific particle approximation on the dual space we are going to use.
To approximate a predictive distribution $\nu_{i|0:i-1}(x_{i})$, the classical particle approximation used in bootstrap particle filtering can be described as  follows:
\begin{list}{
$\bullet$
}{\itemsep=1mm\topsep=2mm\itemindent=-3mm\labelsep=2mm\labelwidth=0mm\leftmargin=9mm\listparindent=0mm\parsep=0mm\parskip=0mm\partopsep=0mm\rightmargin=0mm\usecounter{enumi}}
\setcounter{enumi}{0}
\item sample $X_{i-1}^{(m)} \overset{\text{iid}}{\sim} \nu_{i-1|0:i-1}$, $m = 1, \ldots, N$;
 \item propagate the particles by sampling $X_{i}^{(m)} \sim p_{t}(X_{i-1}^{(m)},\cdot)$, with $p_{t}$ the signal transition density;
 \item estimate $\nu_{i|0:i-1}$ with $\hat \nu_{i|0:i-1} := N^{-1}\sum_{m=1}^N \delta_{X_{i}^{(m)}}$.
\end{list}
For what concerns the use of dual processes, we are going to operate similarly to bootstrap particle filtering but on the dual space. The filtering distributions considered in this work are mixtures of the form
$
 \nu_{i-1|0:i-1}(x_{i-1}) = \sum\nolimits_{\mathbf{m} } w_\mathbf{m} h(x_{i-1}, \mathbf{m})\pi(x_{i-1}) 
$. An estimate of these can be obtained through a particle approximation of the discrete mixing measure, that is we draw $
  \mathbf{m}^{(n)} \overset{\text{iid}}{\sim}  \sum_{\mathbf{m} } w_\mathbf{m} \delta_{\mathbf{m}}$, $n=1,\ldots,N$,
  to obtain
$
\hat \nu_{i-1|0:i-1}(x_{i-1}) := N^{-1}\sum\nolimits_ {n=1}^N h(x_{i-1}, \mathbf{m}^{(n)})\pi(x_{i-1})$.
The natural approximation of $\nu_{i|0:i-1}(x_{i})$ is therefore as follows:
\begin{list}{
$\bullet$
}{\itemsep=1mm\topsep=2mm\itemindent=-3mm\labelsep=2mm\labelwidth=0mm\leftmargin=9mm\listparindent=0mm\parsep=0mm\parskip=0mm\partopsep=0mm\rightmargin=0mm\usecounter{enumi}}
\setcounter{enumi}{0}
\item sample $ \mathbf{m}^{(n)} \overset{\text{iid}}{\sim}  \sum_{\mathbf{m} } w_\mathbf{m} \delta_{\mathbf{m}}$;
 \item propagate the particles by sampling $\mathbf{n}^{(n)} \sim 
 p_{\mm^{(n)},\cdot}(t)$, with $ p_{\mm^{(n)},\cdot}(t)$ the transition probabilities of the dual process;
 \item estimate $\nu_{i|0:i-1}(x_{i})$ with  $\hat \nu_{i|0:i-1}(x_{i}) := N^{-1}\sum\nolimits_ {n=1}^N h(x_i, \mathbf{n}^{(n)})\pi(x_{i})$.
\end{list}

Here some important remarks are in order. 
The above dual particle approximation is a finite mixture approximation of a mixture which can be either finite or infinite. Hence the above strategy can be applied both to filtering given death-like duals but also given general duals on discrete state spaces. The quality of the dual particle approximation, in general, may differ from that obtained through the particle filtering approximation since the particles live on a discrete space in the first case and on a continuous space in the second. This is the object of the following sections, at least for two specific examples.
Finally, the ease of implementation of the two approximations may be very different because simulating from the original Markov process may be much harder than simulating from the dual process. An example is the simulation of Kingman's typed coalescent, immediate as compared to the simulation from \eqref{WF transition}, which would be unfeasible without \cite{JS17}.

\subsection{Cox--Ingersoll--Ross numerical experiments}\label{sec: CIR experiments}

The CIR diffusion admits two different duals:
\begin{list}{
$\bullet$
}{\itemsep=1mm\topsep=2mm\itemindent=-3mm\labelsep=2mm\labelwidth=0mm\leftmargin=9mm\listparindent=0mm\parsep=0mm\parskip=0mm\partopsep=0mm\rightmargin=0mm\usecounter{enumi}}
\setcounter{enumi}{0}
\item the death-like dual given by $D_{t}=(M_{t},\Theta_{t})$, with $M_{t}$ a pure-death process on $\Z_{+}$ with rates $2\sigma^{2}\theta m$ from $m$ to $m-1$ and $\Theta_{t}$ a deterministic process that solves the ODE $\d \Theta_{t}/\d t=-2\sigma^{2}\Theta_{t}(\Theta_{t}-\gamma/\sigma^{2})$, $\Theta_{0}=\theta$. Cf.~\cite{PR14}, Section 3.1.
 \item the B\&D dual $M_{t}$ on $\Z_{+}$ with birth rates from $m$ to $m+1$ given by $\lambda_m = 2\sigma^2 (\delta/2 + m)(\theta - \gamma /\sigma^2)$  and death rates from $m$ to $m-1$ given by $\mu_m = 2 \sigma^2 \theta m$ respectively. Cf.~Proposition \ref{prop: CIR duality}.
\end{list}

Note that the latter is time-homogeneous, the former is not. In general, temporal homogeneity is to be preferred since a direct simulation with a Gillespie algorithm in the inhomogeneous case would require a time-rescaling. However, for this specific case, there is a convenient closed-form expression for the transition density of the first dual, which can be used to simulate for arbitrary time transitions (see the third displayed equation at page 2011 in \cite{PR14}). The second dual, by virtue  of the temporal homogeneity, can be simulated directly using a Gillespie algorithm. This may be slow if the event rate becomes large, but as suggested in Section \ref{sec:CIR} we can see it as a linear B\&D process, and a convenient closed-form expression can be used to simulate arbitrary time transitions.

We compare these two particle approximations with an exact computation of the predictive distribution following \cite{PR14} and to a bootstrap particle filtering approach on the original state space of the signal, which is easy to implement for arbitrary time transitions thanks to the Gamma-Poisson expansion of the CIR transition density (see details in \cite{KKK21}, Section 5).

\begin{figure}[t!]
\begin{center}
\includegraphics[width = \textwidth]{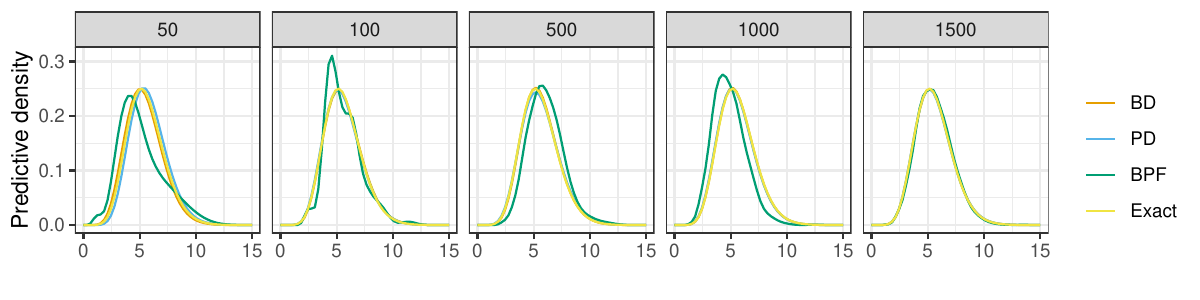}
\captionof{figure}{\footnotesize Comparison of the signal predictive distribution $\hat p(x_{k+1}|y_{1:k})$ obtained through the approximation approach to the death-process dual and the B\&D dual, and through the bootstrap particle filter, with the exact predictive. The number of particles used for the approximations are 50, 100, 500, 1000, 1500 and are indicated in the panel labels. The acronyms are BD: Birth-and-Death, PD: Pure-Death, BPF: Bootstrap Particle Filter. \label{fig:cv_pred_dist_CIR}}
\end{center}
\end{figure}

Figure \ref{fig:cv_pred_dist_CIR} shows the comparison of the above-illustrated strategies, with prediction performed for a forecast time horizon of 0.05. The CIR parameters were specified to $\delta = 11, \sigma = 1, \gamma = 1.1$. 
The starting distribution for the prediction is a filtering distribution for a dataset whose last Poisson observation equals 4, so the starting distribution is a mixture of Gamma densities roughly centred around this point. 
The density estimates for the bootstrap particle filter were obtained from a Gamma kernel density estimator with bandwidth estimated by cross-validation. This is expected to induce a negligible error because the target distribution is a finite mixture of Gamma distributions.

The figure suggests that the bootstrap particle filter is slower to converge to the exact predictive distribution. 
Instead, with only 50 particles, both dual approximations that use a pure-and a B\&D dual (respectively PD and BD in the Figure legend) are already almost indistinguishable from the exact predictive distribution. This shows that accurately approximating the mixing measure on the discrete dual space seems to require fewer particles than approximating the continuous distribution on the original continuous state space. Shorter and longer time horizons than that used in Figure \ref{fig:cv_pred_dist_CIR} were also tested and provided qualitatively similar results.

Next, we turn to investigating the error on the filtering distributions, which combines successive particle approximations. Since the update operation can be performed exactly through \eqref{update operation}, particle filtering using the dual process is conveniently implemented like a bootstrap particle approximation to a Baum-Welch filter with systematic resampling.
We quantify the error on the filtering distributions by measuring the absolute error on the first moment and the standard deviation of the filtering distributions (with respect to the exact computation). We also include the error on the signal retrieval, measured as the absolute difference between the first moment of the filtering distributions and the value of the hidden signal to be retrieved. The mean filtering error is averaged over the second half of the sequence of observations to avoid possible transient effects at the beginning of the observation sequence and estimated over 50 different simulated datasets. The parameter specification is again $\delta = 11, \sigma = 1, \gamma = 1.1$, with a single Poisson observation at each of 200 observation times, and intervals between consecutive observations equal to 0.1.
Figure \ref{fig:filtering_error_CIR} shows that the pure-death particle approximation performs better than the B\&D particle approximation, but the latter performs comparably to the bootstrap particle filter, possibly with a modest advantage.

\begin{center}
\includegraphics[width = \textwidth]{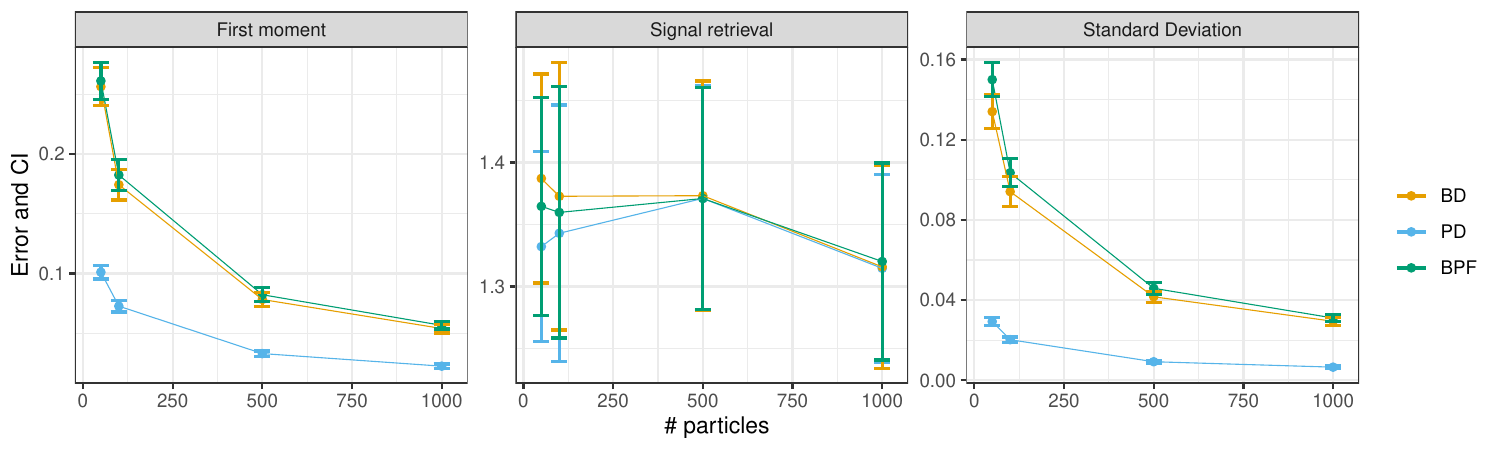}
\captionof{figure}{\footnotesize Mean filtering error as a function of the number of particles for the various particle approximation methods. The error bars represent the confidence interval on the error estimate from the 50 repetitions. The acronyms are BD: Birth-and-Death, PD: Pure-Death, BPF: Bootstrap Particle Filter.\label{fig:filtering_error_CIR}}
\end{center}

\subsection{Wright--Fisher numerical experiments}\label{sec: WF experiments}

The WF diffusion admits two different duals:
\begin{list}{
$\bullet$
}{\itemsep=1mm\topsep=2mm\itemindent=-3mm\labelsep=2mm\labelwidth=0mm\leftmargin=9mm\listparindent=0mm\parsep=0mm\parskip=0mm\partopsep=0mm\rightmargin=0mm\usecounter{enumi}}
\setcounter{enumi}{0}
\item Kingman's typed coalescent with mutation dual, given by a pure-death process on $\Z_{+}^{K}$ with rates $\lambda_{\mm, \mm-\ee_{i}} = m_i(|\boldsymbol \alpha| + | \boldsymbol m| -1)/2$ from $\mm$ to $\mm-\ee_{i}$. Cf.~\cite{PR14}, Section 3.3.
 \item a Moran dual process, given a homogeneous B\&D  process on $\Z_{+}^{K}$ with rates $\lambda_{\mm,\mm-\ee_{i}+\ee_{j}} = m_i(\alpha_j + m_j)/2$ from $\mm$ to $\mm-\ee_{i}+\ee_{j}$. Cf.~Proposition \ref{prop: WF duality}.
\end{list}

Here both processes are temporally homogeneous and can thus be easily simulated using a Gillespie algorithm, with the only caveat that the simulation can be inefficient when the infinitesimal rates are large. Similar to the CIR case, there is a closed-form expression for the transition probabilities in the first case, which can be used for simulation purposes for arbitrary time transitions (see Theorem 3.1 in \cite{PRS16}). Unlike the one-dimensional CIR case, handling this expression is challenging in the multi-dimensional WF case, with significant numerical stability issues raised by the need to compute the sum of alternated series with terms that can both overflow and underflow. In \cite{KKK21}, these hurdles were addressed using arbitrary precision computation libraries and careful re-use of previous computations applicable when data is equally spaced.
The Gillespie simulation strategy presents no such restriction and may be significantly faster when the event rates remain low.

As mentioned in Section \ref{sec:WF}, no closed-form expression is available for the Moran dual and the Gillespie algorithm approach is the main option, likely resulting in a slow algorithm. Alternatively, as argued in Section \ref{sec:WF}, we can approximate the Moran dual process by a finite population Wright--Fisher chain, with the quality of approximation increasing with the population size.
The interest in this approximation is that the event rate for the latter is lower than for the Moran process. This is related to the fact that weak convergence of a sequence of WF chains to a WF diffusion occurs when time is rescaled by a factor of $N$ (cf.~\cite{KT81}, Sec.~15.2.F), whereas a Moran model whose individual updates occur at the times of a Poisson process with rate 1, needs a rescaling by a factor $N^{2}$ to obtain a similar convergence.  In other words, in order to establish weak convergence to the diffusion, time must be measured in units of $N$ generations in the WF chain and in units of $N^{2}$ generations in the Moran model. See discussion in Section \ref{sec:WF}.
For this reason, the resulting Gillespie simulation is expected to be faster using a WF chain approximation to the Moran model. 

The above considerations also suggest another possibility. Since the Moran process converges weakly to a Wright--Fisher diffusion, the latter could also be used as a possible approximation instead of a WF chain. In this case, it is possible to sample directly from \eqref{WF transition} for arbitrary  time transitions using the algorithm in \cite{JS17}. Hence we would be using a WF diffusion to approximate the dual Moran transitions in \eqref{WF propagation}. 

A standard bootstrap particle filter performed directly on the Wright--Fisher diffusion state space also crucially relies on the algorithm of \cite{JS17} for the prediction step, without which approximate sampling recipes from the transition density would be needed. 

In Figure \ref{fig:cv_pred_dist_WF}, we compare prediction strategies for a  WF diffusion with $K=4$ types using:
\begin{list}{
$\bullet$
}{\itemsep=1mm\topsep=2mm\itemindent=-3mm\labelsep=2mm\labelwidth=0mm\leftmargin=9mm\listparindent=0mm\parsep=0mm\parskip=0mm\partopsep=0mm\rightmargin=0mm\usecounter{enumi}}
\setcounter{enumi}{0}
\item the closed-form transition of the pure death dual (``Exact'' in Fig.~\ref{fig:cv_pred_dist_WF} legend);
\item an approximation of the pure death dual using a Gillespie algorithm (``PD'');
 \item an approximation of the Moran dual using a Gillespie algorithm (``BD Gillespie Moran'');
 \item a WF chain approximation of the Moran dual using a Gillespie algorithm (``BD Gillespie WF'');
 \item a WF diffusion approximation of the Moran dual using \cite{JS17} (``BD diffusion WF'');
 \item a bootstrap particle filtering approximation using \cite{JS17} (``Bootstrap PF'').
\end{list}

In Figure \ref{fig:cv_pred_dist_WF}, prediction was performed for a forecast time horizon equal to 0.1, with WF parameters $\aalpha = (3, 3, 3, 3)$. 
The starting distribution for the prediction is a filtering distribution for a dataset whose last multinomial observation is equal to $(4, 0, 9, 2)$ (so the starting distribution is a mixture of Dirichlet distributions roughly centred around $(4/15, 0, 9/15, 2/15)$). 
Various values for parameter $\aalpha$ were also tested and provided results qualitatively similar to Figure \ref{fig:cv_pred_dist_CIR}.
The density estimates for the bootstrap particle filter are obtained from a Dirichlet kernel density estimator with bandwidth estimated by cross-validation (using Julia package KernelEstimators \url{https://github.com/panlanfeng/KernelEstimator.jl}). This is expected to induce a negligible error because the target distribution is a finite mixture of Dirichlet distributions.
Figure \ref{fig:cv_pred_dist_WF} shows that among these particle approximations of $p(x_{k+1}|y_{1:k})$, the Wright--Fisher diffusion approximation of the Moran dual seems to converge slowest, followed by the bootstrap particle filter, whereas the other strategies based on the dual process converge quickly to the exact distribution.

\begin{center}
\includegraphics[width = \textwidth]{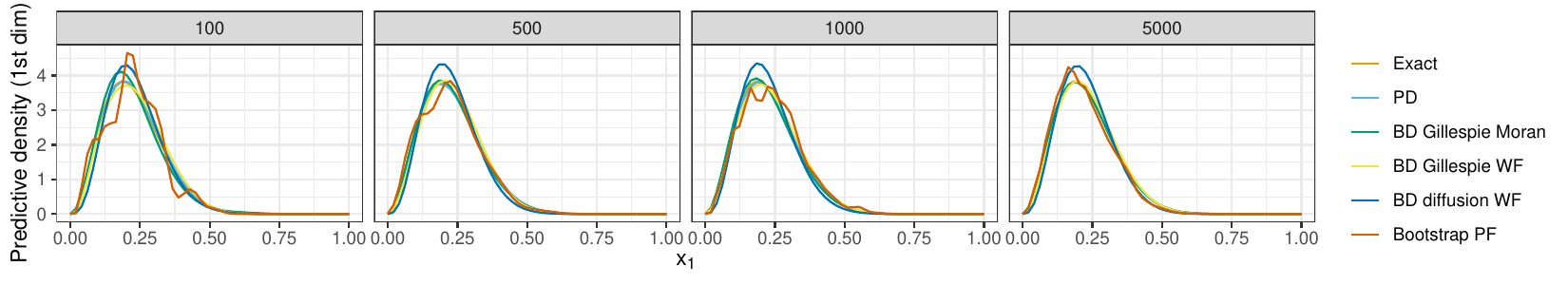}
\captionof{figure}{\footnotesize Convergence of the WF predictive distribution (only the first dimension) with the number of particles for the various particle approximations. The acronyms are PD: Pure-Death, BD: Birth-and-Death, WF: Wright-Fisher, PF: Particle Filter. \label{fig:cv_pred_dist_WF}}
\end{center}

Figure \ref{fig:filtering_error_WF} evaluates the filtering error for a  WF process with $K = 3$ and parameters $\aalpha = (1.1, 1.1, 1.1)$, given 20 categorical observations collected at each time, over 10 collection times spaced by intervals equal to 1. We consider increasing numbers of particles and use 100 replications to estimate the error.
The figure shows that the particle approximation of the pure death dual process using the closed-form transition exhibits better performance.
The bootstrap particle approximation has the fastest improvement relative to increasing the number of particles. Overall, the Moran dual performs better or comparably to bootstrap particle filtering. 

\begin{center}
\includegraphics[width = \textwidth]{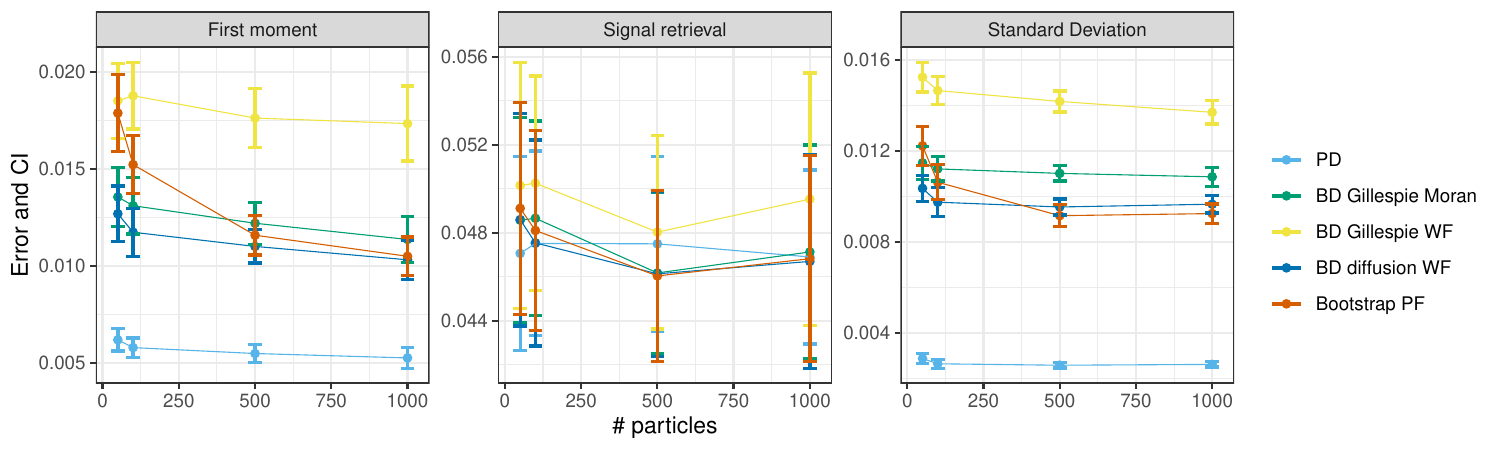}
\captionof{figure}{\footnotesize Mean filtering error as a function of the number of particles for the various particle approximation methods. The error bars represent the confidence interval on the error estimate from the 100 repetitions. The acronyms are PD: Pure-Death, BD: Birth-and-Death, WF: Wright-Fisher, PF: Particle Filter.\label{fig:filtering_error_WF}}
\end{center}


\section{Concluding remarks}

We have provided conditions for filtering diffusion processes on multidimensional continuous spaces which avoid computations on the state space of the forward process when a dual process given by a discrete Markov chain is available. Motivated by certain diffusion models for which only duals with a countable state space are known (e.g., B\&D-like duals for WF diffusions with selection), we have investigated the performance of filtering based on a B\&D dual for the CIR diffusion and based on a Moran process dual for the WF diffusion. All approximation methods proposed appear to be valuable strategies, despite resting on different simulation schemes. The optimal strategy is bound to depend on the application at hand, together with several other details like the interval lengths between data collection times, and possibly be constrained by which of these tools are available. For example, the transition function of coupled WF diffusions \cite{Aea19} is not available, whereas a discrete dual was found in \cite{Fea21}. Overall, approximate filtering using B\&D-like duals may perform better or comparably to bootstrap particle filtering, with the advantage of operating on a discrete state space. The computational effort for each of these strategies  is also bound to depend on a series of factors the identification of which is beyond the scope of this contribution.

The code to reproduce the analyses illustrated above will be made available in the Supporting Material and is based on the package freely available at {\small \url{https://github.com/konkam/DualOptimalFiltering.jl}}.


\section{Acknowledgements}

The authors are grateful to two anonymous referees for carefully reading the manuscript and for providing constructive suggestions that led to improving the paper. They also gratefully acknowledge insightful conversations with Omiros Papaspiliopoulos on performing particle filtering on the dual space. \\
The last author is partially supported by MUR, PRIN project 2022CLTYP4.

\end{document}